\documentclass[11pt,a4paper, twoside]{amsart}

\usepackage{amscd, mathrsfs}
\usepackage[all]{xy}

\usepackage[dvips]{graphicx}
\usepackage{wrapfig}

\theoremstyle{plain}
\numberwithin{equation}{section}
\newtheorem{thm}{Theorem}[section]
\newtheorem{prop}[thm]{Proposition}
\newtheorem{cor}[thm]{Corollary}
\newtheorem{lem}[thm]{Lemma}

\theoremstyle{definition}
\newtheorem{dfn}[thm]{Definition}

\newtheorem{rmk}[thm]{Remark}

\def\rank{\mathop{\mathrm{rank}}\nolimits}

\def\Hom{\mathop{\mathrm{Hom}}\nolimits}

\def\<{{\langle}}
\def\>{{\rangle}}

\def\Aut{\mathop{\mathrm{Aut}}\nolimits}
\def\+{\mathop{\oplus}\nolimits}
\def\FM{\mathop{\mathscr{F}\mspace{-8mu}\mathscr{M}}\mspace{-6mu}}

\newcommand{\mf}[1]{{\mathfrak{#1}}}
\newcommand{\mi}[1]{{\mathit{#1}}}
\newcommand{\bb}[1]{{\mathbb{#1}}}
\newcommand{\mca}[1]{{\mathcal{#1}}}
\newcommand{\mr}[1]{{\mathrm{#1}}}

\title[FM transformations and Atkin-Lehner involutions]{Fourier-Mukai transformations on K3 surfaces with $\rho =1$ and Atkin-Lehner involutions}
\author{Kotaro Kawatani}
\date{\today}
\address{Department of Mathematics, 
Nagoya University, 
Furocho, Chikusa-ku, Nagoya, Japan}
\email{kawatani@math.nagoya-u.ac.jp}

\subjclass[2010]{}

\begin{document}
\maketitle


\begin{abstract}
We show that there is a surjection from the Fourier-Mukai transformations on projective K3 surfaces with the Picard number $\rho (X)=1$ to so called to the group of Atkin-Lehner involutions. 
This was expected in Hosono-Lian-Oguiso-Yau's paper. 
\end{abstract}

\section{Introduction}

\subsection{Terminologies and backgrounds}
Let $D(M)$ be the bounded derived category of coherent sheaves on a projective manifold $M$. 
In this article a projective manifold $M'$ is said to be a \textit{Fourier-Mukai partner} of $M$ if there is an equivalence $\Phi \colon D(M') \to D(M)$. 
Any equivalence $\Phi \colon D(M_1) \to D(M_2)$ between Fourier-Mukai partners of $M$ is said to be a \textit{Fourier-Mukai transformation} on $M$. 
The number of isomorphic classes of Fourier-Mukai partners of $M$ is said to be the \textit{Fourier-Mukai number} of $M$. 
It is conjectured that the Fourier-Mukai number of any projective manifold is finite by Kawamata in \cite{Kaw02}.
For instance the conjecture holds for curves (For example see \cite[mainly in Chapter 5]{Huy}) and surfaces (\cite{BM01} and \cite{Kaw02}). 
And also, the conjecture holds for abelian varieties (essentially \cite{Orl} and independently \cite{Fav}). 

\subsection{The study of \cite{HLOYb}}

The main interest of this paper is the relation, which is predicted by \cite[Remark in page 25]{HLOYb}, between Atkin-Lehner involutions and the Fourier-Mukai number of projective K3 surfaces $X$ with $\rho (X) =1$. 
In the following we briefly recall the study of \cite{HLOYb}.

Suppose that $X$ is a projective K3 surface with $\mr{NS}(X)  = \bb Z L_X$ and with $L_X ^2 = 2d$. 
The numerical Grothendieck group $\mca N(X) $ of $X$ has the Mukai (or Euler) paring $\<-,-\>$ with the signature $(2, 1)$. 
Then as was shown by Dolgachev, the isometry group of $O^+ (\mca N(X))/ \pm \mr{id}$ is isomorphic to Atkin-Lehner modular group $\mr{AL}_d$ of level $d$ (See also Definition \ref{AL}). 

Now recall that any autoequivalence on $D(X)$ induces an isometry on $\mca N(X)$.  
Then we have a representation from $\Aut (D(X)) $ to $\mr{AL}_d$. 
By virtue of \cite[Corollary 3]{HMS} we see the image of this representation is Fricke modular group $\mr{Fr}_d$  which is a subgroup of $\mr{AL}_d $. 
Surprisingly \cite{HLOYb} showed that the index $[\mr{AL}_d: \mr{Fr}_d]$ is equal to the Fourier-Mukai number of $X$. 
Furthermore they predicts that all Atkin-Lehner involutions are obtained from Fourier-Mukai transformations $\Phi :D(Y) \to D(X)$ on $X$. 

\subsection{Our results}

In our main theorem, Theorem \ref{surjective}, we show that Hosono-Lian-Oguiso-Yau's conjecture holds. 
To formulate our results transparently we introduce the notion of the groupoid $\FM_X$ consisting of Fourier-Mukai transformations on $X$ (See Definition \ref{groupoid}). 
Moreover we construct an explicit correspondence between cosets of $\mr{AL}_d/ \mr{Fr}_d$ and Fourier-Mukai partners of $X$. 
Namely we have the surjective functor 
\[
\tilde \rho\colon \FM_X \to O^+(\mca N(X)),
\]
where $O^+(\mca N(X))$ is the orientation preserving isometry group of $\mca N(X)$. 

Now recall \cite[Corollary 3]{HMS}: 
There is a surjection 
$$\rho \colon \Aut (D(X)) \to O^+_{\text{Hodge}}(H^*(X, \bb Z))$$
(See also Theorem \ref{HMSthm}). 
If we restrict $\tilde \rho$ to $\Aut(D(X))$, this gives the representation $\rho$.  Hence our theorem can be regarded as a slight generalization of \cite[Corollary 3]{HMS}. 

\subsection*{Acknolegement}
The author was partially supported by Grant-in-Aid for Scientific Research (S), No 22224001.

\section{Preliminaries}
\subsection{Induced morphisms on $\bb H$}

To discuss the relation between Fourier-Mukai transformations and Atkin-Lehner modular group, we recall the representation of Fourier-Mukai transformations to $\mi{PSL}_2(\bb R)$, the automorphism group of the upper half plain $\bb H$.

We first consider the numerical Grothendieck group 
\[
\mca N(X) = H^0(X, \bb Z) \+ \mr{NS}(X) \+ H^4(X, \bb Z). 
\]
The \textit{Mukai paring} (or \textit{Euler paring}) on $\mca N(X)$ is given by 
\[
\< r\+ c \+ s, r' \+ c' \+ s' \> = c c' - r s' - s r'. 
\]
By the Hodge index theorem, the index of the Mukai paring is $(2, \rho(X))$. 

For objects $E \in D(X)$ we put $v(E) = ch(E) \sqrt{td_X}$ and call it the \textit{Mukai vector} of $E$. 
One can check that $v(E) = r\+ c\+ s \in \mca N(X)$ and see that $r = \rank E$, $c = c_1 (E)$ and $s = \chi (X,E) - \rank E$ by using Riemann-Roch theorem.

Let $\mf {D}^+(X)$ be one of the connected component 
\[
\{ [v] \in \bb P(\mca N(X) \otimes \bb C)  | v^2 =0, v \bar v >0 \}
\]
containing $[\exp(\sqrt{-1}\omega )]$ where $\omega$ is an ample divisor. 
As is well-known $\mf D^+(X)$ is isomorphic to the tube domain $\mr{NS}(X)_{\bb R} \times C^+(X)$ where $C^+(X)$ is the positive cone:
\[
\mr{NS}(X) _{\bb R} \times C^+(X) \ni (\beta, \gamma) \mapsto [\exp(\beta + \sqrt{-1} \gamma)] \in \mf{D}^+(X). 
\]
We remark that if $\rho (X) =1$, $\mf{D}^+(X)$ is canonically isomorphic to the upper half plain $\bb H$:
\[
\bb H \ni u+\sqrt{-1}v \mapsto [\exp\big( (u +\sqrt{-1}v)L\big)] \in \mf{D}^+(X),
\]
where $L$ is an ample basis of $\mr{NS}(X)$.

Now suppose that $X$ and $Y$ are K3 surfaces with $\rho(X)= \rho(Y)=1$ and $\Phi : D(Y) \to D(X)$ is an equivalence. 
We put the degree of $X$ and $Y$ by $2d$. 
Since $\Phi$ induces the orientation preserving isometry $\Phi ^N \colon \mca N(Y) \to \mca N(X)$ by \cite[Theorem 2]{HMS}, we obtain the morphism 
\[
\Phi _* \colon \mf{D}^+(Y) \to \mf D^+(X). 
\]
Since both $\mf{D}^+(Y)$ and $\mf{D}^+(X)$ are $\bb H$, we obtain the automorphism on $\bb H$ by using the canonical isomorphism:
\[
\Phi_*(u_Y + \sqrt{-1}v_Y) = u_X + \sqrt{-1}v_X. 
\]
This automorphism was calculated by the author \cite[Lemmas 3.1 and 3.2]{Kaw12}. 
To explain these lemmas, we set the following:
\[
v(\Phi (\mca O_{y})) = r_X \+ n_X L_X \+ s_X \text{ and } v(\Phi ^{-1}(\mca O_{x})) = r_Y \+ n_Y L_Y \+ s_Y. 
\]
Here $x \in X$ and $y \in Y$ are closed points. 
Then $\Phi _*$ is given as follows:

\begin{prop}[{\cite[Lemmas 3.1 and 3.2]{Kaw12}}]\label{Kawlem}
Let $\Phi \colon D(Y) \to D(X)$ be an equivalence between projective K3 surfaces with $\rho =1$ and let $\Phi _*$ be the induced automorphism on $\bb H$. 
\begin{enumerate}
\item We have $r = r_X=r_Y$. Moreover if $r_X =0$, then 
\[\Phi _*(u_Y+ \sqrt{-1}v_Y) = x _Y + m +\sqrt{-1} v_Y \]
for some $m \in \bb Z$. 
\item Suppose that $r \neq 0$. Then $\Phi _*$ is given by 
\[
\Phi_*(u_Y  + \sqrt{-1}v_Y) = \frac{1}{d|r|} \cdot \frac{-1}{(u_Y+ \sqrt{-1}v_Y)- \frac{n_Y}{r} } +\frac{n_X}{r}. 
\]
\end{enumerate}
\end{prop}

Original proof is written in terms of Bridgeland stability conditions on $X$. 
So, for the convenience of readers we write the proof. 

\begin{proof}
Recall $v(\mca O_x) = 0\+ 0\+ 1$. Then we see  
\[
-r_X = \< v(\Phi (\mca O_{y})), v(\mca O_{x})  \> = \< v(\mca O_{y}), v(\Phi^{-1}(\mca O_{x}))  \> = -r_Y. 
\]
Thus we see $r_X = r_Y$. 
Moreover, if $r_X =0$ then one can see that $Y$ is isomorphic to $X $ and that the equivalence $\Phi$ is numerically equivalent to $\otimes( m L_X)$ for some $m \in \bb Z$ via an isomorphism $f \colon Y\to X$ (The details are in \cite[Lemma 3.1]{Kaw12}).

The second assertion follows from \cite[Lemma 3.2]{Kaw12}. 
We recall the proof. 

One can easily see 
\[
\Phi ^H(\exp( (u_Y + \sqrt{-1} v_Y) L_Y)) = \lambda \exp((u_X + \sqrt{-1} v_X) L_X)
\]
for some $\lambda \in \bb C^*$.

Put $\beta_X + \sqrt{-1}\omega _X= u_X L_X + \sqrt{-1}v_X L_X$ (respectively $Y$). 
We can define the function $Z_{(\beta_X, \omega_X)}: \mca N(X) \to \bb C$, which is usually called a \textit{central charge}:
\begin{eqnarray*}
Z_{(\beta_X, \omega_X)} (E)	 &:=& \< \exp(\beta_X+ \sqrt{-1}\omega_X, v(E)) \>\\
							&=& \frac{v(E)^2}{2 r } + \frac{r}{2} 
								\Big(\omega_X + \sqrt{-1 }(\frac{c}{r} - \beta_X) \Big)^2,	
\end{eqnarray*}
where $v(E)  = r \+ c \+ s$. 
Then we see 
\begin{eqnarray*}
\lambda	&=&	- \< \Phi ^H (\exp(\beta _Y + \sqrt{-1} \omega _Y)), v(\mca O_{x})  \> \\
			&=&	-\< \exp(\beta _Y+ \sqrt{-1}\omega _Y) ,  v(\Phi ^{-1}(\mca O_{x}) ) \> \\
			&=& - Z_{(\beta _Y, \omega _Y)} (\Phi ^{-1}(\mca O_{x})), 
\end{eqnarray*}
and
\begin{eqnarray*}
-1		&=&	\< \exp(\beta _Y + \sqrt{-1}\omega _Y), v (\mca O_{y}) \>\\
		&=&	\< \Phi ^H(\exp(\beta _Y + \sqrt{-1}\omega _Y)), v(\Phi (\mca O_{y})) \> \\
		&=& \lambda \cdot Z_{(\beta _X, \omega _X)} (\Phi (\mca O_{y})). 
\end{eqnarray*}
Thus we have 
\[
1 = Z_{(\beta _Y, \omega _Y)}(\Phi ^{-1}(\mca O_{x}))  \cdot  Z_{(\beta _X, \omega _X)} (\Phi (\mca O_{y})) . 
\]

Since $v(\Phi (\mca O_{y}))^2= v(\Phi ^{-1}(\mca O_{x}))^2 =0$, we have 
\[
Z_{(\beta _Y, \omega _Y)} (\Phi ^{-1}(\mca O_{x})) = \frac{r}{2} \Big( v_Y+ \sqrt{-1} \big( \frac{n_Y}{r} - u_Y \big)   \Big)^2 L_Y^2
\]
and 
\[
Z_{(\beta _X, \omega _X)} (\Phi (\mca O_{y})) = \frac{r}{2} \Big( v_X+ \sqrt{-1} \big( \frac{n_X}{r} - u_X \big)   \Big)^2 L_X^2. 
\]
Since $L_X^2 = L_Y ^2 = 2d$ we see 
\begin{equation*}
(u_X- \frac{n_X}{r}) + \sqrt{-1}v_X = \frac{\pm 1}{d |r|} \cdot \frac{1}{(u_Y- \frac{n_Y}{r_Y})+\sqrt{-1}v_Y}.  \label{eqn1}
\end{equation*}
Since the left hand side is in the upper half plain $\bb H$, the imaginary part of the left hand side is positive. 
Hence we have
\[
(u_X- \frac{n_X}{r}) + \sqrt{-1}v_X = \frac{-1}{d |r|} \cdot \frac{1}{(u_Y- \frac{n_Y}{r})+\sqrt{-1}v_Y}.  
\]
Thus we have finished the proof. 
\end{proof}

\subsection{Atkin-Lehner and Fricke involutions}

In this section we recall the Atkin-Lehner involutions and Fricke involutions.  
As usual we put 
\[
\Gamma _0(d) = \{ \begin{pmatrix}\alpha & \beta \\ \gamma & \delta  \end{pmatrix}  \in \mi{PSL}_2 (\bb Z) | \gamma \in d\bb Z   \}. 
\]

For integers $s, d \in \bb Z$ we define the symbol $s || d$ by 
\begin{equation}
s || d \stackrel{\mr{def}}{\iff}  s | d \text{ and } \gcd (s, \frac{d}{s})=1. \label{s||d}
\end{equation}
Suppose that $s || d$. 
We put 
\[
W_s  = \{ \frac{1}{\sqrt{s}} \begin{pmatrix} \alpha & \beta \\ \gamma & \delta \end{pmatrix}  
\begin{pmatrix} s & 0 \\ 0 & 1  \end{pmatrix} \in \mi{PSL}_2(\bb R)   |  \gamma \in \frac{d}{s} \bb Z\text{ and } \delta \in  s  \bb Z \}. 
\]
$W_s$ is also given as 
\[
W_s  = \{ \begin{pmatrix} \alpha \sqrt{s}  & \frac{\beta}{\sqrt{s}} \\ \gamma \frac{d}{s} \sqrt{s} & \delta \sqrt{s} \end{pmatrix}  \in \mi{PSL}_2(\bb R)  | \alpha, \beta, \gamma \text{ and }\delta \in \bb Z \}. 
\]
In particular we see $W_1 = \Gamma _0(d)$. 

For cosets $W_s$ one can easily check the following:
\begin{lem}[\cite{CN79}]
Each $W_s$ is in the normalizer of $\Gamma _0(d)$ in $\mi{SL}_2(\bb R)$. 
In addition the coset classes $W_s$ and $W_{s'}$ satisfies the following rule:
\[
W_s ^2 = W_1, W_s W_{s'}  =W_{s'}W_s  = W_{s*s '}, 
\]
where $s * s' = \frac{s s'}{\gcd (s,s ')^2}$
\end{lem}

\begin{dfn}\label{AL}
We put
\[
\mr{AL}_d := \bigsqcup_{s || d} W_s
\text{ and }
\mr{Fr}_d  := W_1 \sqcup  W_d. 
\]
We call $\mr{AL}_d$ and $\mr{Fr}_d$ respectively the \textit{Atkin-Lehner modular group} and the \textit{Fricke modular group}. 
\end{dfn}

\begin{rmk}
By the above lemma we see both sets $\mr{AL}_d$ and $\mr{Fr}_d$ have group structures. 
Moreover, $\mr{AL}_d$ is the abelian normalizer group of $\Gamma _0(d)$ in $\mi{PSL}_2(\bb R)$. 
Since $W_s W_d = W_{\frac{d}{s}}$, the coset decomposition of $\mr{AL}_d/ \mr{Fr}_d$ is given by $$ \mr{AL}_d/ \mr{Fr}_d = \bigsqcup_{s || d} (W_s \sqcup W_{\frac{d}{s}}).$$
\end{rmk}

\subsection{An explicit construction of Fourier-Mukai partners of $X$}

In this subsection we recall the work of \cite{HLOYa} which is an explicit construction of Fourier-Mukai partners of $X$ with $\mr{NS}(X) = \bb Z L$. 
Put $L ^2 =2d$ as usual. 

We set the set $P_d$ by 
\[
P_d =\{ r \in \bb N | r ||d \}/\sim
\]
where $r_1 \sim r_2$ if and only if $r_1 = r_2$ or $r_1 = \frac{d}{r_2}$. 

\begin{thm}[{\cite[Theorem 2.1]{HLOYa}}]\label{HLOYathm}
Let $X$ be a projective K3 surface with $\mr{NS}(X)  = \bb Z L$. Put $L^2 =2d$.  
There is a one to one correspondence between $P_d$ and the set $\mr{FM}_X$ of isomorphic classes of Fourier-Mukai partners of $X$:
\[
P_d \ni r \mapsto M_L(r \+L \+ \frac{d}{r}) \in \mr{FM}_X. 
\]
Here $M_L(r\+ L \+s)$ is the fine moduli space of $\mu_L$-stable sheaves with Mukai vector $r\+ L \+s$. 
\end{thm}

\subsection{Lattices and modular groups}

The aim of this subsection is to recall Dolgachev's theorem. 

Let $N_d = \bb Z e_0 \+ \bb Z \ell \+ \bb Z e_4$ be the abstract lattice with the intersection matrix $\Sigma$ where  
\[
\Sigma = \begin{pmatrix} 0 & 0 & -1 \\ 0 & 2d & 0 \\ -1 & 0 & 0  \end{pmatrix}. 
\]
Let $O(N_d)$ be the orthogonal group of $N_d$:
\[
O(N_d) = \{ g \in \mi{GL}_3(\bb Z) | {}^{t}g \Sigma g = \Sigma \}. 
\]
Put $O^+(N_d)$ be the subgroup consisting of $g \in O(N_d)$ which preserves the orientation of positive $2$ plane in $N _{\bb R} = N_d \otimes _{\bb Z} \bb R$. 
Since the intersection form is non-degenerate, we see $N_d \subset N_d^{\vee} = \Hom(N_d, \bb Z)$ in $N_{\bb R}$. 
Hence $g \in O(N_d)$ induces the isometry on the discriminant lattice $A_{N_d} = N_d^{\vee}/ N_d$ with respect to the natural quadratic form. 
We define $O(N_d)^*$ by the kernel of the morphism $O(N_d) \to O (A_{N_d})$ and define $O^+(N_d)^* = O^+(N_d) \cap O(N_d)^*$.

Now put $\mi{SO}^+(N_d) = \{  g\in  O^+(N_d) | \det g =1  \}$. 
Then $\mi{SO}^+(N_d)$ is isomorphic to $\mi{PSL}(2, \bb R)$ by the following morphism 
\[
R:\mi{PSL}(2, \bb R) \to \mi{SO}^+(N_d), 
\begin{pmatrix} \alpha & \beta \\ \gamma & \delta \end{pmatrix}
\mapsto
\begin{pmatrix} 
\delta ^2 & 2 \gamma \delta  & \frac{1}{d}\gamma ^2 \\ 
\beta \delta & \alpha \delta + \beta \gamma & \frac{1}{d} \alpha \gamma \\
d \beta ^2 & 2d \alpha \beta & \alpha ^2
\end{pmatrix}. 
\]
Then we have the following sequence of morphisms:
\[
q: O^+(N_d) \to O^+(N_d) / \pm{\mr{id}_{N_d}} \stackrel{\sim}{\to} \mi{SO}^+(N_d) \stackrel{R^{-1}}{\to} \mi{PSL}(2, \bb R). 
\]

In this situation Dolgachev proves the following:

\begin{thm}[{\cite[Theorem 7.1 and Remark 7.2]{Dol}}]\label{Dolthm}
The image $q(O^+(N_d)^*)$ of $O^+(N_d)^*$ is the Fricke modular group $\mr{Fr}_d$ and the image $q (O^+(N_d))$ of $O^+(N_d)$ is the Atkin-Lehner modular group. 
\end{thm}

\subsection{Modular groups and autoequivalences on $D(X)$}

By using Dolgachev's theorem and the theorem of \cite{HMS} (below), we discuss the relation between the Fricke modular group and the autoequivalence group $\Aut (D(X))$. 

Since $X$ is a K3 surface, any autoequivalence $\Phi \in \Aut (D(X))$ induces the Hodge isometry $\Phi^H$ of the integral cohomology ring $H^*(X, \bb Z)$. 

\begin{thm}[{\cite[Corollary 3]{HMS}}]\label{HMSthm}
Let $X$ be a projective K3 surface (not necessary Picard rank $1$). 
Then the morphism 
\[
\rho : \Aut (D(X)) \to O_{\mi{Hodge}}(H^*(X, \bb Z)), \Phi \mapsto \Phi ^H
\]
is surjective to the orientation preserving isometry group $O^+_{\mi{Hodge}}(H^*(X, \bb Z))$. 
\end{thm}
In particular we obtain the isometry $\Phi^H |_{\mca N(X)}$ on the numerical Grothendieck group $\mca N(X)$. 
Suppose $\rho (X) =1$ with $L_X^2 =2d$. 
Since $\mca N(X)$ is canonically isomorphic to the abstract lattice $N_d$. 
Thus we have the morphism:
\begin{equation}
M:\Aut (D(X)) \stackrel{\rho}{\to} O^+(H^*(X, \bb Z)) \stackrel{ | _{N_d}}{\to} O^+(N_d) \stackrel{q}{\to} \mi{PSL}_2(\bb R).  \label{M}
\end{equation}

By combining the above two theorems, we obtain the following proposition. 

\begin{prop}\label{DHMS}
The morphism $M$ is surjective to the Fricke modular group $\mr{Fr}_d$. 
\end{prop}

\begin{proof}
We first show $\mr{Im}(M) \subset \mr{Fr}_d$. 
By Theorem \ref{Dolthm}, it is enough to show that $\Phi^H|_{\mca N(X)}$ is in $\< O^+(\mca N(X))^*,  \pm\mr{id}\> \subset O^+(N_d)$. 

Since $\rho (X) =1$, the restriction $\Phi ^H | _{T(X)}$ to the transcendental lattice $T(X) (\subset H^*(X, \bb Z))$ is $\pm \mr{id}_{T(X)}$ by the result of Oguiso \cite[Lemma 4.1]{Ogu}.
Moreover since a single shift $[1]$ is in the kernel of $M$, we may assume ${\Phi}^H  | _{T(X)} = \mr{id}_{T(X)} $ by composing a single shift. 
We note that $\Phi ^H | _{T(X)}$ induces the identity on the discriminant lattice of $T(X)$. 
Since the discriminant lattice $A_{T(X)} = T(X)^{\vee}/T(X)$ of $T(X)$ is  canonically isomorphic the discriminant lattice $A_{\mca{N}(X)}$ of $\mca N(X)$, 
$\Phi^H | _{\mca{N}(X)}$ is in $O^+(\mca N(X))^*$. 
Hence $\mr{Im}(M)$ contained in $\mr{Fr}_d$ by Theorem \ref{Dolthm}.  

Conversely we show $\mr{Fr}_d \subset \mr{Im}(M)$. 
Take an arbitrary $\varphi \in O^+(N(X))^*$. 
Then $\varphi \+ \mr{id}_{T(X)}$ extends to the isometry $\tilde{\varphi}$ on the hole lattice $H^*(X, \bb Z)$. 
Since $\varphi$ preserves the orientation of $\mca N(X)$, $\tilde{\varphi}$ also preserves the orientation of $H^*(X, \bb Z)$. 
Since the natural representation $\Aut (D(X)) \to O^+(H^*(X, \bb Z))$ is surjective by Theorem \ref{HMSthm}, there is an autoequivalence $\Phi$ such that $\Phi ^H = \tilde{\varphi}$. 
\end{proof}


\section{Main result and the proof}

We first remark that the set of all Fourier-Mukai transformations has naturally a groupoid structure. 
Namely we define the following:

\begin{dfn}\label{groupoid}
Let $M$ be a projective manifold. 
We define the groupoid $\FM_M$ as follows: 
\begin{itemize}
\item Objects of $\FM_M$ consist of Fourier-Mukai partners of $X$:
\[
\mr{Ob}(\FM_M) = \{ W:\text{projective manifold}|\exists \Phi \colon D(W) \stackrel{\sim}{\to} D(M)\}. 
\]
\item Morphisms in $\FM_M$ are Fourier-Mukai transformations beween them:
\[
\mr{Mor}_{\FM_M}(W, W') = \{ \Phi : D(W) \stackrel{\cong}{\to} D(W'), \text{FM transformations on }M   \}
\]
\end{itemize}
Since any Fourier-Mukai transformation gives $\Phi:D(W) \to D(W')$ a morphism of $\FM_M$, we write as $\Phi \in \FM_M$. 
We call $\FM_M$ the \textit{groupoid of Fourier-Mukai transformations} on $M$ (or shortly \textit{Fourier-Mukai groupoid}). 
\end{dfn}

Now suppose that $M=X$ is a projective K3 surface with $\rho (X) =1$. 
Let $\Phi :D(Y) \to D(Y')$ be in $\FM_X$. 
Since the numerical Grothendieck groups of $Y$ and $Y'$ are canonically isomorphic to the abstract lattice $N_d$, the equivalence $\Phi$ induces the orientation preserving isometry $\Phi^N$ on $N_d$. 
Namely we have the functor from the groupoid to the isometry group of $N_d$ by using these canonically isomorphisms:
\[
\rho' :\FM_X \to O^+(N_d), \Phi \mapsto \Phi^N. 
\]

\begin{rmk}
By composing the morphism $q: O^+(N_d) \to \mi{PSL}_2(\bb R)$, we obtain the following functor
\[
M = q \circ \rho ' : \FM_X \to \mi{PSL}_2(\bb R), \Phi \mapsto q(\Phi^N). 
\]
Since the restriction of $q \circ \rho '$ to $\Aut (D(X))$ is the same as the group morphism $M: \Aut (D(X)) \to \mi{PSL}_2(\bb R)$, we put $M = q \circ \rho '$ by abusing notations. 
By the definition of the functor $M$, we see $M(\Phi)$ is just the linear fractional transformation $\Phi_*$ ginve in Proposition \ref{Kawlem}
\end{rmk}

\begin{thm}\label{surjective}
Let $\FM_X$ be the Fourier-Mukai groupoid on a K3 surface $X$ with $\rho (X)=1$. 
We put $\mr{NS}(X) = \bb Z L$ with $L^2 = 2d$. 
\begin{enumerate}
\item The functor $M \colon \FM_X \to \mr{AL}_d$ is surjective. Namely for any $\varphi \in \mr{AL}_d$, there exists a Fourier-Mukai transformation $\Phi \colon D(Y) \to D(X)$ in $\FM_X$ such that $M(\Phi) =\varphi$. 
\item For $\Phi\colon D(Y) \to D(Y') \in \FM_X$, $Y$ is isomorphic to $Y'$ if and only if $M(\Phi ) \in \mr{Fr}_d$. 
\end{enumerate}
\end{thm}

\begin{proof}
Recall Proposition \ref{DHMS}. 
By this proposition, it is enough to show that for any $s || d$, there is a Fourier-Mukai transformation $\Phi \colon D(Y) \to D(X)$ such that $M(\Phi) \in W_s$.

For the integer $s$, we put $r = \frac{d}{s}$ and take an isotropic Mukai vector $v\in \mca N(X)$ as $v= r \+ L_X \+ s$. 
Then there exists the fine moduli spaces $M_L(r \+ L \+ s)$ of $\mu$-stable sheaves with Mukai vector $v = r\+ L \+ s$ since $\gcd (r, L_X^2, s )=1$. 
We put $Y = M_L(r \+ L \+ s)$ and let $\mca E$ be the universal family of the moduli space. 
We claim that the Fourier-Mukai transformation $\Phi _{\mca E}: D(Y) \to D(X)$ satisfies $M(\Phi_{\mca E} ) \in W_{s}$ where 
\[
\Phi _{\mca E}(-) \colon D(Y) \to D(X), \Phi_{\mca E}(-)= \bb R \pi_{X *}(\mca E \stackrel{\bb L}{\otimes} \pi _{Y}^*(-)) . 
\]

Put $v(\Phi _{\mca E}^{-1}(\mca O_x))  = r \+ n L_Y \+ s'$. 
By Proposition \ref{Kawlem} the linear fractional transformation $M(\Phi _{\mca E})$ is given by the following matrix:
\[
M(\Phi_{\mca E}) = 
\begin{pmatrix}
\sqrt{\frac{d}{r}} & -\sqrt{\frac{r}{d}} \frac{r + d n}{r^2} \\
r \sqrt{\frac{d}{r}} & -n \sqrt{\frac{d}{r}}
\end{pmatrix}.
\]
To prove our claim it is enough to show that $\frac{r+ dn}{r^2}$ is an integer. 
To show this, we consider the inverse Fourier-Mukai transformation $\Phi _{\mca E}^{-1}: D(X) \to D(Y)$. 
Then the matrix $M(\Phi_{\mca E}^{-1})$ is given by 
\[
M(\Phi_{\mca E}^{-1}) = 
\begin{pmatrix}
n\sqrt{\frac{d}{r}} & -\sqrt{\frac{r}{d}} \frac{r + d n}{r^2} \\
r \sqrt{\frac{d}{r}} & -\sqrt{\frac{d}{r}}
\end{pmatrix}.
\]
Since $(\Phi_{\mca E}^{-1})^N = \pm R\circ M(\Phi _{\mca E}^{-1}) \in O^+(N_d)$, all coefficient of the $3 \times 3$ matrix of $R \circ M(\Phi _{\mca E}^{-1})$ should be integers. 
By focusing $(2,1)$ component of $R \circ M(\Phi _{\mca E}^{-1})$ we see that \[
\sqrt{\frac{r}{d}} \frac{r + dn}{r^2} \times \sqrt{\frac{d}{r}} = \frac{r + dn}{r^2}
\]
is an integer. 
This gives the proof of the first assertion.  

Now we prove the second assertion. 
Let $\Phi \colon D(Y_1) \to D(Y_2) \in \FM_X$. 
By Theorem \ref{HLOYathm} we can assume $Y_i \stackrel{f_i}{\cong}  M_i = M_{L_X}(r_i \+ L_X \+ s_i)$ ($i=1,2$) with $r_i || d$. 
By using these isomorphisms we get the Fourier-Mukai transformation
\[
\Phi ' = f_{2*} \circ \Phi  \circ f_{1*}^{-1}\colon D(M_1) \to D(M_2). 
\]
We note that $M(\Phi ')= M(\Phi)$ since $M(f_{i*}) = \mathrm{id}$ ($i=1,2$). 

By the proof of the first assertion we see that there is an equivalence $\Psi_i:D(M_i)\to D(X)$ such that $M(\Phi _i) \in W_{\frac{d}{r_i}}$. 
Then we get the following commutative diagram:
\[
\begin{CD}
D(M_1) @>\Phi'>> D(M_2) \\
@V\Psi _1VV	@VV\Psi_2V \\
D(X) @>>\Psi_2 \cdot \Phi' \cdot \Psi _1^{-1}> D(X)
\end{CD}
\]
Since $\tilde \Phi  = \Psi_2 \circ \Phi' \circ \Psi _1^{-1}$ is an autoequivalence, $M(\tilde \Phi) \in W_1 \sqcup W_d$. 
In particular by composing an equivalence $T \in \Aut (D(X))$ so that $M(T) \in W_d$, we can assume that $M(\tilde \Phi ) \in W_1$. 
Since $\Phi \in W_1 \sqcup W_d$, we have to consider two cases: 
If $\Phi \in W_1$ then we have 
\[
W _{\frac{d}{r_1}} W_1 W_{\frac{d}{r_2}} = W _{s_1}\cdot  W_{s_2}  = W_1. 
\]
Hence we see $s_1 = s_2$ and $r_1 =r_2$. Thus $Y_1 = Y_2$.

If $\Phi \in W_{d}$ then 
\[
W _{\frac{d}{r_1}} W_d W_{\frac{d}{r_2}} = W _{r_1}\cdot  W_{s_2}  = W_1. 
\]
Thus we see $r_1 = s_2$. 
Since $M(r \+ L \+ s) \cong M(s \+ L \+r)$ by Theorem \ref{HLOYathm}, we see $Y_1 \cong Y_2$. 
Thus we have proved the second assertion. 
\end{proof}

By combining Proposition \ref{DHMS}, we obtain the following corollary.  

\begin{cor}\label{main result}

The following functor is surjective$:$
\[
\tilde {\rho} = |_{N_{d}} \circ \rho : \FM_X \ni \Phi \mapsto \Phi ^N \in O^+(N_d). 
\]
\end{cor}

\begin{proof}
Recall that he functor $M \colon \FM_X \to \mi{PSL}_2(\bb R)$ factors through $O^+(N_d)$. 
Hence we obtain the following commutative diagram:
\[
\xymatrix{
\FM_X \ar[rr]^{M} \ar[rd]_{\tilde{\rho}} &		& \mr{AL}_d \\
			&O^+(N_d)\ar[ur]_q&
}
\] 
Since $M$ and $q$ are surjective by Theorems \ref{Dolthm} and \ref{surjective}, we see $\tilde \rho$ is also surjective. 
\end{proof}

\begin{rmk}
Corollary \ref{main result} can be regarded as the generalization of Theorem \ref{HMSthm}. 
Furthermore to generalize our result to arbitrary Picard rank cases, we have to find some canonical identification of numerical Grothendieck groups between Fourier-Mukai partners. 
\end{rmk}

\end{document}